\documentclass[12pt,reqno]{amsart}
\usepackage{fullpage}
\usepackage{colonequals}
\usepackage[table]{xcolor}
\usepackage{times}
\usepackage{amsmath,amssymb,amsthm,url}
\usepackage[utf8]{inputenc}
\usepackage{mathrsfs}
\usepackage[english]{babel}
\usepackage{bbm}
\usepackage{enumerate}
\usepackage{bm}
\usepackage{graphicx}
\usepackage[colorlinks=true, pdfstartview=FitH, linkcolor=blue, citecolor=blue, urlcolor=blue]{hyperref}
\usepackage{tabstackengine}
\stackMath

\newtheorem{thm}{Theorem}[section]
\newtheorem{lem}[thm]{Lemma}
\newtheorem{prop}[thm]{Proposition}
\newtheorem{defn}[thm]{Definition}

\newtheorem{cor}[thm]{Corollary}
\newtheorem{rem}[thm]{Remark}

\newtheorem{question}[thm]{Question}
\newtheorem{prob}[thm]{Problem}

\newcommand{\Z}{\mathbb{Z}}
\newcommand{\F}{\mathbb{F}}

\AtBeginDocument{%
	\def\MR#1{}
}

\title{The EKR-module property of pseudo-Paley graphs of square order}
\author{Shamil Asgarli}
\address{Department of Mathematics and Computer Science \\ Santa Clara University \\ 500 El Camino Real \\ USA 95053}
\email{sasgarli@scu.edu}
\author{Sergey Goryainov}
\address {School of Mathematical Sciences \\ Hebei Key Laboratory of Computational Mathematics and Applications \\ Hebei Normal University\\Shijiazhuang  050024\\ P.R. China}
\email{sergey.goryainov3@gmail.com}
\author[03]{Huiqiu Lin}
\address{Department of Mathematics\\ East China University of Science and Technology\\ Shanghai 200237\\ P.R. China}
\email{huiqiulin@126.com}
\author{Chi Hoi Yip}
\address{Department of Mathematics \\ University of British Columbia \\ 1984 Mathematics Road \\ Vancouver  V6T 1Z2 \\ Canada}
\email{kyleyip@math.ubc.ca}
\date{\today}

\begin{document}

\maketitle

\begin{abstract}
We prove that a family of pseudo-Paley graphs of square order obtained from unions of cyclotomic classes satisfies the Erd\H{o}s-Ko-Rado (EKR) module property, in a sense that the characteristic vector of each maximum clique is a linear combination of characteristic vectors of canonical cliques. This extends the EKR-module property of Paley graphs of square order and solves a problem proposed by Godsil and Meagher. Different from previous works, which heavily rely on tools from number theory, our approach is purely combinatorial in nature. The main strategy is to view these graphs as block graphs of orthogonal arrays, which is of independent interest.
\end{abstract}

\section{Introduction}

Throughout the paper, let $p$ be an odd prime, $q$ a power of $p$. Let $\F_q$ be the finite field with $q$ elements, $\F_q^+$ be its additive group, and $\F_q^*=\F_q \setminus \{0\}$ be its multiplicative group. 

Given an abelian group $G$ and a connection set $S \subset G \setminus \{0\}$ with $S=-S$, the {\em Cayley graph} $\operatorname{Cay}(G,S)$ is
the undirected graph whose vertices are elements of $G$, such that two vertices $g$ and $h$ are adjacent if and only if $g-h \in S$. A {\em clique} in a graph $X$ is a subset of vertices of $X$ such that any two of them are adjacent. For a graph $X$, the {\em clique number} of $X$, denoted $\omega (X)$, is the size of a maximum clique of $X$.

Given any graph $X$ for which we can describe its canonical cliques (that is, typically cliques with large size and simple structure), we can ask whether $X$ has any of the following three related Erd\H{o}s-{K}o-{R}ado (EKR) properties; see Section~\ref{subsect:EKR} for more background and connection to other EKR-type results.

\begin{itemize}
    \item EKR property: the clique number of $X$ equals the size of canonical cliques.
    \item EKR-module property: the characteristic vector of each maximum clique in $X$ is a $\mathbb{Q}$-linear combination of characteristic vectors of canonical cliques in $X$. 
    \item strict-EKR property: each maximum clique in $X$ is a canonical clique. 
\end{itemize}

Consider the Paley graph $P_{q^2}$ which is the Cayley graph defined on the additive group of $\F_{q^2}$, with the connection set being the set of squares in $\F_{q^2}^*$. Clearly, the subfield $\F_q$ forms a clique. Moreover, $a\F_q+b$ also forms a clique for each $a, b \in \F_{q^2}$ where $a$ is a nonzero square. Such square translates of $\F_q$ are the canonical cliques \cite[Section 5.9]{GM} in this example. Blokhuis proved that these are precisely the maximum cliques in $P_{q^2}$.

\begin{thm}[\cite{Blo84}] \label{EKR}
Let $q$ be an odd prime power. The Paley graph $P_{q^2}$ satisfies the strict-EKR property.
\end{thm}

Godsil and Meagher \cite[Section 5.9]{GM} call Theorem~\ref{EKR} the EKR theorem for Paley graphs. Theorem~\ref{EKR} was first proved by Blokhuis \cite{Blo84}. Extensions and generalizations of Theorem~\ref{EKR} can be found in \cite{BF91, Szi99, NM, AY2, AY}. A Fourier analytic approach was recently proposed in \cite[Section 4.4]{Yip4}. While we have at least three different proofs of Theorem~\ref{EKR}, all known proofs relied heavily on advanced tools such as the polynomial method over finite fields.

Instead, in this paper, we will follow a purely combinatorial approach. Although we are not able to give a simple proof of Theorem~\ref{EKR}, we prove that a weaker version of Theorem~\ref{EKR} extends to a larger family of Cayley graphs, namely Peisert-type graphs; see Theorem~\ref{thm: main}. Peisert-type graphs were first introduced explicitly in \cite[Definition 1.1]{AY}, but can be dated back to \cite{BWX}; see the discussion before Corollary~\ref{cor: srg}. See Lemma~\ref{Ptexample} for examples of such graphs: Paley graphs, Peisert graphs, and their generalizations. The following definition differs slightly from \cite[Definition 1.1]{AY} which had a slightly stronger hypothesis $m\leq \frac{q+1}{2}$.

\begin{defn}[Peisert-type graphs]\label{defn:peisert-type}
Let $q$ be an odd prime power. Let $S \subset \F_{q^2}^*$ be a union of $m \leq q$ cosets of $\F_q^*$ in $\F_{q^2}^*$ such that $\F_q^* \subset S$, that is, 
\begin{equation}\label{coset}
S=c_1\F_q^* \cup c_2\F_q^* \cup \cdots \cup c_m \F_q^*.
\end{equation}
Then the Cayley graph $X=\operatorname{Cay}(\F_{q^2}^+, S)$ is said to be a Peisert-type graph of type $(m,q)$. A clique in $X$ is called a \emph{canonical clique} if it is the image of the subfield $\F_q$ under an affine transformation.
\end{defn}

In this paper, we discuss the three different notions of EKR properties mentioned above in the context of Peisert-type graphs. For any Peisert-type graph $X$ of type $(m,q)$, we always assume $m \leq q$. Note that the subfield $\F_q$ forms a clique in $X$. In particular, $\omega(X)\geq q$. The hypothesis that $m\leq q$ is crucial for the equality $\omega(X)=q$ to hold as can be seen from the proof of Theorem~\ref{thm:peisert-type-chromatic}. Hence, Peisert-type graphs have the EKR-property. We refer to other known EKR-related properties of Peisert-type graphs in Section~\ref{subsect:PTgraph}.

Blokhuis' theorem already implies that Paley graphs of square order enjoy the EKR-module property. In their book, Godsil and Meagher ask for an algebraic proof of this statement \cite[Problem 16.5.1]{GM}, which motivates this work.

Our main result in the present paper answers this problem for a larger family of Cayley graphs:

\begin{thm}\label{thm: main}
Peisert-type graphs satisfy the EKR-module property.
\end{thm}

The main ingredient in the proof is the following connection between Peisert-type graphs and orthogonal arrays, which is of independent interest. 

\begin{thm}\label{thm: PT=OA}
Each Peisert-type graph of type $(m,q)$ can be realized as the block graph of an orthogonal array $OA(m,q)$. Moreover, there is a one-to-one correspondence between canonical cliques in the block graph and canonical cliques in a given Peisert-type graph.
\end{thm}

We remark that the idea of viewing certain Cayley graphs geometrically has appeared in the past; see for example \cite[Construction 5.2.1]{NM} and \cite[Section 4.2]{AY} for related discussion. However, Paley graphs and block graphs of orthogonal arrays are often treated independently; see for example \cite[Chapter 10]{GR01}, \cite[Chapter 5]{GM}, and \cite[Section 5]{AFMNSR21}. The present paper is the first to make an explicit connection between Peisert-type graphs and orthogonal arrays: Theorem~\ref{thm: PT=OA} allows us to treat them in a uniform manner. We mention the following four additional applications below.

It is known that the block graph of an orthogonal array is strongly regular. Thus, Theorem~\ref{thm: PT=OA} also implies the same conclusion for the Peisert-type graphs. We remark that Peisert-type graphs in fact form a subfamily of a well-known family of strongly regular Cayley graphs defined on finite fields due to Brouwer, Wilson, and Xiang~\cite{BWX}: the connection set is a union of semi-primitive cyclotomic classes of $\F_{q^2}$. However, their proof heavily relied on the fact we can compute semi-primitive Gauss sums {\em explicitly} using Stickelberger's theorem and its variants; see \cite[Proposition 1]{BWX} and \cite[Corollary 3.6]{AY2}. We will see that Theorem~\ref{thm: PT=OA} can be proved using a purely combinatorial argument, thus giving an elementary proof of the corollary below.

\begin{cor}\label{cor: srg}
A Peisert-type graph of type $(m,q)$ is strongly regular with parameters
$
(q^2, m(q - 1), (m - 1)(m - 2) + q - 2, m(m - 1))
$
and eigenvalues
$k=m(q-1)$ (with multiplicity $1$), $-m$ (with multiplicity  $q^2-1-k$) and
$q-m$ (with multiplicity $k$). In particular, a Peisert-type graph of type $(\frac{q+1}{2},q)$ is a pseudo-Paley graph.
\end{cor}

Paley graphs are known to have a close connection with Paley's construction on Hadamard matrices \cite{Paley}. Recently, Adm et al. \cite{AFMNSR21} introduced the notion of weakly Hadamard matrices and studied weakly Hadamard diagonalizable graphs (see Definition~\ref{defn: WHD} and \ref{defn: WHDG}). In particular, they showed that Paley graphs of square order are weakly Hadamard diagonalizable \cite[Theorem 5.9]{AFMNSR21}. The following theorem generalizes their result.

\begin{thm}\label{thm: PTWHD}
Peisert-type graphs are weakly Hadamard diagonalizable.
\end{thm}

Recall that the {\em chromatic number} of a graph $X$, denoted $\chi(X)$, is the smallest number of colors needed to color the vertices of $X$ so that no two adjacent vertices share the same color. We remark that one can prove the original EKR theorem using the (fractional) chromatic number of Kneser graphs \cite[Theorem 7.8.1]{GR01}. It is known that the chromatic number is bounded below by the clique number, that is, $\omega(X) \leq \chi(X)$. Broere, D\"oman, and Ridley \cite{BDR88} showed that if $d>1$ and $d \mid (q+1)$, then both the chromatic number and the clique number of $GP(q^2, d)$ is $q$; the graph $GP(q^2, d)$ is the $d$-Paley graph of order $q^2$, whose precise definition can be found in Section~\ref{subsect:PTgraph} below. The converse of this result was proved by Schneider and Silva \cite[Theorem 4.7]{SS15}; a stronger converse was proved recently in \cite{Yip4}. The following theorem computes both the chromatic and the clique number of all Peisert-type graphs, hence extending the first result since $GP(q^2, d)$ with $d\mid (q+1)$ is a Peisert-type graph by Lemma~\ref{Ptexample}.

\begin{thm}\label{thm:peisert-type-chromatic}
Let $X$ be a Peisert-type graph of order $q^2$. Then $\omega(X)=\chi(X)=q$. In particular, $X$ has the EKR property.
\end{thm}

Theorem~\ref{thm: PT=OA} also implies the following corollary on the strict-EKR property of a special family of Peisert-type graphs, which includes Sziklai's theorem on generalized Paley graphs \cite{Szi99} in case of small edge density. Although the corollary is slightly weaker than \cite[Corollary 4.1]{AY}, the proof of this weaker result is much simpler. In contrast, the proof of the stronger result \cite[Corollary 4.1]{AY} relied on the characterization of the number of directions determined by the graph of a function \cite{Ball03} over finite fields, which was built on \cite{LR73,BBS,BBBSS}. 

\begin{cor}\label{cor: EKRmsmall}
If $q>(m-1)^2$, then all Peisert-type graphs of type $(m,q)$ satisfy the strict-EKR property. In particular, if $d > \frac{q+1}{\sqrt{q}+1}$ and $d\mid (q+1)$, then the $d$-Paley graph $GP(q^2, d)$ has the strict-EKR property.
\end{cor}

It is natural to examine when a Peisert-type graph $X$ enjoys the strict-EKR property. While we do not have a general answer to this problem, we exhibit an infinite family of Peisert-type graphs which fail to satisfy the strict-EKR property in Section~\ref{sect:counterexamples}. The following theorem shows that the condition $q>(m-1)^2$ in Corollary~\ref{cor: EKRmsmall} is sharp when $q$ is a square.

\begin{thm}\label{thm: counterexample}
Let $q$ be an odd prime power which is not a prime. Then there exists a Peisert-type graph $X$ of order $q^2$ such that $X$ fails to have the strict-EKR property. In particular, if $q$ is a square, then there exists a Peisert-type graph $X$ of type $(\sqrt{q}+1,q)$ which fails to have the strict-EKR property.
\end{thm}

\subsection*{Outline of the paper} In Section~\ref{sect:prelim} we include more background and provide further motivation for our work. In Section~\ref{sect:peisert-orthogonal-array} we prove Theorem~\ref{thm: PT=OA} and deduce Corollary~\ref{cor: srg}, Theorem~\ref{thm: PTWHD}, Theorem~\ref{thm:peisert-type-chromatic}, and Corollary~\ref{cor: EKRmsmall}. We explore the EKR-module property of Peisert-type graphs and prove Theorem~\ref{thm: main} in Section~\ref{sect:ekr-module}. Section~\ref{sect:counterexamples} discusses the strict-EKR property of Peisert-type graphs and contains the proof of Theorem~\ref{thm: counterexample} and an explicit counterexample $GP^*(81, 10)$. Finally, Section~\ref{sect:open-problems} considers two open problems related to the present work.

\section{Preliminaries}\label{sect:prelim}

The structure of this background section is as follows. We briefly overview EKR-type results in \ref{subsect:EKR}, Peisert-type graphs in \ref{subsect:PTgraph}, strongly regular graphs in \ref{subsect:strongly-regular}, orthogonal arrays in \ref{subsect:orthogonal-array}, and weakly Hadamard matrices in \ref{subsect:weakly-hadamard}.

\subsection{EKR-type results}\label{subsect:EKR}

The classical Erd\H{o}s-Ko-Rado theorem \cite{EKR61} classified maximum intersecting families of $k$-element subsets of $\{1, 2, ..., n\}$ when $n\geq 2k+1$. Since then, EKR-type results refer to understanding maximum intersecting families in a broader context, and more generally, classifying extremal configurations in other domains. The book \cite{GM} by Godsil and Meagher provides an excellent survey on the modern algebraic approaches to proving EKR-type results for permutations, set systems, orthogonal arrays, and so on.  

The EKR-type problems related to a transitive permutation group $G$ can be reformulated in terms of the EKR properties of cocliques of the derangement graph $\Gamma(G)$, or equivalently, the cliques of the complement. Once we define canonical cocliques (or cliques), we can discuss the EKR properties of $G$ after identifying $G$ with $\Gamma(G)$. The EKR-module property was first formally defined by Meagher \cite{Mea19} in this context: a permutation group $G$ naturally acts on the vector space $W$ spanned by the characteristic vectors of canonical cliques, which makes $W$ a $G$-module. 

Each finite $2$-transitive group enjoys the EKR property \cite{MSP16}. Meagher and Sin \cite{MS21} recently showed that all finite $2$-transitive groups have the EKR-module property. However, the strict-EKR property does not hold for permutations groups in general; recently, Meagher and Razafimahatratra \cite{MR22} have shown that the general linear group $GL(2,q)$ is such a counterexample. We remark that our results are of similar flavor, although in our context of Peisert-type graphs, the corresponding vector space $W$ does not carry a natural module structure. However, we shall remark that the definition of the EKR-module property (even for permutation groups) does not need the additional $G$-module structure. 

In general, the \emph{module method} (see \cite[Section 4]{AM15}) refers to the strategy of proving that a graph $\Gamma$ satisfies the strict-EKR property in two steps:
\begin{itemize}
    \item show that $\Gamma$ satisfies the EKR-module property
    \item show that the EKR-module property implies the strict-EKR property
\end{itemize}
As an example of the module method, \cite[Theorem 4.5]{AM15} provides a sufficient condition for the second step above for 2-transitive permutation groups.

The EKR-type problems discussed in this paper are about cliques of Peisert-type graphs.
The algebraic graph theory approach to prove Theorem~\ref{EKR} suggested by \cite[Section 16.5]{GM} is precisely the module method. In the present paper, we confirm the first step for all Peisert-type graphs in Theorem~\ref{thm: main}. We show that the second step fails for an infinite family of Peisert-type graphs, and we discuss concrete counterexamples in Section~\ref{sect:counterexamples}. We remark that in~\cite[Lemma 5.6]{AFMNSR21}, a proof of the EKR-module property of Paley graphs of square order is given without using Theorem~\ref{EKR}. 

\subsection{Peisert-type graphs}\label{subsect:PTgraph} 

Let $d>1$ be a positive integer. If $q \equiv 1 \pmod {2d}$, the {\em $d$-Paley graph} on $\F_q$ \cite{SC, LP}, denoted $GP(q,d)$, is the Cayley graph  $\operatorname{Cay}(\F_q^+,(\F_q^*)^d)$, where $(\F_q^*)^d$ is the set of $d$-th powers in $\F_q^*$. Sziklai \cite{Szi99} showed that if $d>1$ and $d \mid (q+1)$, then $GP(q^2,d)$ has the strict-EKR property. Note that a Peisert-type graph is simply the union of copies of generalized Paley graphs $GP(q^2,q+1)=\operatorname{Cay}(\F_{q^2}^+,\F_q^*)$: while each copy obviously has the strict-EKR property, it is not clear whether their union would preserve this property.

Peisert \cite{WP2} introduced a new family of graphs to classify self-complementary symmetric graphs, and this family of graphs are now known as Peisert graphs. The {\em Peisert graph} of order $q=p^r$, where $p$ is a prime such that $p \equiv 3 \pmod 4$ and $r$ is even, denoted $P^*_q$, is the Cayley graph $\operatorname{Cay}(\F_{q}^{+}, M_q)$ with $M_q= \{g^j: j \equiv 0,1 \pmod 4\},$ where $g$ is a primitive root of the field $\F_q$. Mullin \cite[Chapter 8]{NM} conjectured that if $q \equiv 3 \pmod 4$, then the Peisert graph with order $q^2$ has the strict-EKR property. In \cite[Theorem 1.4]{AY}, we confirmed her conjecture when $q=p^n$ and $p>8.2n^2$.

Mullin introduced the notion of generalized Peisert graphs; see \cite[Section 5.3]{NM}. 
\begin{defn}
Let $d$ be a positive even integer, and $q$ a prime power such that $q \equiv 1 \pmod {2d}$. The {\em $d$-th power Peisert graph of order $q$}, denoted $GP^*(q,d)$, is the Cayley graph $\operatorname{Cay}(\F_q^+, M_{q,d})$, where
$$
M_{q,d}=\bigg\{g^{dk+j}: 0\leq j \leq \frac{d}{2}-1, k \in \Z\bigg\},
$$
and $g$ is a primitive root of $\F_q$.
\end{defn}

The following lemma further motivates the definition of Peisert-type graphs (Definition~\ref{defn:peisert-type}) by showing that they unify the previously studied graphs.

\begin{lem} [{\cite[Lemma 2.10]{AY}}] \label{Ptexample}
The following families of Cayley graphs are Peisert-type graphs:
\begin{itemize}
    \item Paley graphs of square order;
    \item Peisert graph with order $q^2$, where $q \equiv 3 \pmod 4$;
    \item Generalized Paley graphs $GP(q^2,d)$, where $d \mid (q+1)$ and $d>1$;
    \item Generalized Peisert graphs $GP^*(q^2,d)$, where $d \mid (q+1)$ and $d$ is even.
\end{itemize}
\end{lem}

In \cite[Theorem 1.2]{AY}, we showed that each maximum clique in a Peisert-type graph of type $(m,q)$ with $m \leq \frac{q+1}{2}$ is an affine $\F_p$-subspace. In particular, every Peisert-type graph of type $(m,p)$ with $m \leq \frac{p+1}{2}$ has the strict-EKR property. To deduce that maximum cliques of a general Peisert-type graph $X$ have the subfield structure, that is, $X$ has the strict-EKR property, we borrowed character sum estimates over subspaces. We found that the strict-EKR property holds under some conditions on the connection set of $X$ and the characteristic $p$ \cite[Theorem 1.3]{AY}.

We expected that directly extending Blokhuis' and Sziklai's proofs to a general Peisert-type graph is difficult \cite[Remark 2.16]{AY}. In fact, we speculated that there might be an infinite family of Peisert-type graphs which fail to have the strict-EKR property and gave a few counterexamples of small size in \cite[Example 2.18]{AY}. We confirm our prediction in Theorem~\ref{thm: counterexample}.

\subsection{Strongly regular graphs}\label{subsect:strongly-regular}
We first recall the definition of strongly regular graphs.
\begin{defn}[Strongly regular graph]
If $X$ is a $k$-regular graph with $n$ vertices, such that any two adjacent vertices have $\lambda$ common neighbors, and any two distinct non-adjacent vertices have $\mu$ common neighbors, then $X$ is a {\em strongly regular graph} with parameters $(n, k, \lambda, \mu)$.
\end{defn}

It is well-known that Paley graphs are strongly regular; see for example \cite[Theorem 5.8.1]{GM}.

\begin{thm} \label{Paleysrg}
If $q \equiv 1 \pmod 4$, then the Paley graph $P_q$ is a strongly regular graph with parameters $(q,\frac{q-1}{2},\frac{q-5}{4},\frac{q-1}{4})$. Moreover, the eigenvalues of $P_q$ are $\frac{q-1}{2}$ (with multiplicity 1), $\frac{1}{2}(-1 \pm \sqrt{q})$ (each with multiplicity $\frac{q-1}{2}$).
\end{thm}

Weng, Qiu, Wang, and Xiang \cite{WQWX} introduced the definition of pseudo-Paley graphs.

\begin{defn}[Pseudo-Paley graph] \label{defn: ppg}
A {\em pseudo-Paley graph} is a strongly regular graph with the same parameters $(n,k, \lambda, \mu)$ as the Paley graph $P_q$ for some $q$.
\end{defn}

A clique $C$ in a regular graph is called \emph{regular} if every vertex that is not in $C$ has the same number of neighbors in $C$. The following lemma gives an upper bound on the clique number of a strongly regular graph, and shows that a maximum clique whose size agrees with the given upper bound must be regular; see for example \cite[Proposition 1.3.2]{BCN89}.

\begin{lem} [Delsarte-Hoffman bound]\label{HoffmanBound}
Suppose that $X$ is a strongly regular graph with parameters $(n,k,\lambda,\mu)$ and smallest eigenvalue $-m$.
Let $C$ be a clique in $X$. Then $|C| \le 1+\frac{k}{m}$, with equality if and only if every vertex that is not in $C$ has the same number of neighbors (namely $\frac{\mu}{m}$) in $C$. 
\end{lem}

\subsection{Block graphs of orthogonal arrays and their EKR properties}\label{subsect:orthogonal-array}
In this subsection, we recall basic terminology about orthogonal arrays and revisit the related EKR properties.

An {\em orthogonal array} $OA(m, n)$ is an $m \times n^2$ array with entries from an $n$-element set $T$
with the property that the columns of any $2 \times n^2$ subarray consist of all $n^2$ possible pairs. The {\em block graph of an orthogonal array} $OA(m, n)$, denoted $X_{OA(m,n)}$,  is defined to be the graph
whose vertices are columns of the orthogonal array, where two columns are adjacent if there exists a row in which they have the same entry. Let $S_{r,i}$ be the set of columns of $OA(m, n)$ that have the entry $i$ in row $r$. These sets are cliques, and since each element of the $n$-element set $T$ occurs
exactly $n$ times in each row, the size of $S_{r,i}$ is $n$ for all $i$ and $r$. These cliques
are called the \emph{canonical cliques} in the block graph $X_{OA(m,n)}$. A simple combinatorial argument shows that the block graph of an orthogonal array is strongly regular.

\begin{thm}[{\cite[Theorem 5.5.1]{GM}}]\label{thm: OAsrg}
If $OA(m, n)$ is an orthogonal array where $m < n+1$, then
its block graph $X_{OA(m,n)}$ is strongly regular with parameters
$$
(n^2, m(n - 1), (m - 1)(m - 2) + n - 2, m(m - 1));
$$
the eigenvalues of $X_{OA(m,n)}$ are $m(n - 1)$, $n-m$, and $-m$ with multiplicities $1$, $m(n - 1)$, and $(n - 1)(n + 1 - m)$, respectively.
\end{thm}

Combining Lemma~\ref{HoffmanBound} and Theorem~\ref{thm: OAsrg}, we see that the clique number of $X_{OA(m,n)}$ is at most $1+\frac{m(n-1)}{m}=n$, which is equal to the size of the canonical clique. Thus, block graphs of orthogonal arrays have the EKR-property. It is known that when $n>(m-1)^2$, the block graph of $OA(m,n)$ has the strict-EKR property; see \cite[Corollary 5.5.3]{GM}. We include a short proof for the sake of completeness, especially because it will lead to a simple and self-contained proof of Corollary~\ref{cor: EKRmsmall}.

\begin{thm}\label{thm: EKROA}
Let $X=X_{OA(m,n)}$ be the block graph of an orthogonal array $OA(m, n)$ with
$n >(m - 1)^ 2$. Then $X$ has the strict-EKR property: the only maximum cliques in $X$ are the columns that have entry $i$ in row $r$ for some $1 \leq i \leq n$ and $1 \leq r \leq m$.
\end{thm}

\begin{proof}
The case $m=1$ is trivial. Next we assume $m \geq 2$. Clearly, the set of columns that have entry $i$ in row $r$ forms a clique of size $n$. Next, we consider a clique $C$ that is not of this form, and we will show $|C|\leq (m-1)^2<n$. Then we can conclude that the only cliques of size $n$ in $X$ are canonical.

Let $C$ be a non-canonical clique in the block graph of an orthogonal array $OA(m, n)$. We can relabel the entries to be $\{0, 1, \ldots, n-1\}$. By translating entries of each row modulo $n$, we may assume, without loss of generality, that the column $c_0$ with all zeros is in $C$. Then other columns in $C$ must each have exactly one zero entry. Let $D_i$ be the set of columns in $C\setminus \{c_0\}$ having a zero entry in the $i$-th row. Without loss of generality, we assume that $|D_1| \leq |D_2| \leq \cdots \leq |D_m|$. By assumption, we have $0<|D_{m-1}|\leq |D_m|$. Let $c$ be a column in $D_{m-1}$, then for each $1 \leq j \leq m-2$, there is at most one column in $D_m$ that shares the same entry with $c$ in row $j$. And if $j \in \{m-1,m\}$, there is no column in $D_m$ that shares the same entry with $c$ in row $j$. Therefore, $|D_m| \leq m-2$. It follows that \[|C| \leq 1+\sum_{j=1}^{m}|D_j| \leq 1+m|D_m| \leq (m-1)^2<n.\qedhere\]
\end{proof}

When $n \leq (m-1)^2$, Theorem~\ref{thm: EKROA} no longer holds \cite[Section 5.5]{GM}. However, it is known that a weaker statement is true, namely: the block graph $X_{OA(m,n)}$ always satisfies the EKR-module property (for example, this follows from \cite[Theorem 5.5.5]{GM}). We will give a detailed explanation for this fact in Section~\ref{sect:ekr-module}, and use it to establish our main result.

An orthogonal array $OA(m,n)$ is called {\em extendible} if it occurs as the first $m$ rows of an $OA(m+1,n)$. The following theorem characterizes extendible orthogonal arrays.

\begin{thm}[{\cite[Theorem 10.4.5]{GR01}}] \label{thm: OAchi}
An orthogonal array $OA(m,n)$ is extendible if and only if its block graph $X_{OA(m,n)}$ has chromatic number $n$.
\end{thm}

\subsection{Weakly Hadamard diagonalizable graphs}\label{subsect:weakly-hadamard}

Recall that a \emph{Hadamard matrix} is a square matrix with entries $1$ or $-1$ such that any two columns are mutually orthogonal. There are several open problems about the structure of the Hadamard matrices. One of them is determining the existence of a Hadamard matrix of a particular order. One classical construction dates back to Paley \cite{Paley} using quadratic residues over a finite field, which eventually motivated the definition of the Paley graph; see \cite{GJ} for a historical discussion.

A graph $\Gamma$ is called \emph{Hadamard diagonalizable} if the Laplacian of $\Gamma$ can be diagonalized by a Hadamard matrix \cite{BFK11}. Recently, Adm et al. \cite{AFMNSR21} studied a larger class of graphs which contains some families of strongly regular graphs. In order to present this result, they introduced a broader class of matrices which include Hadamard matrices.

\begin{defn}[\cite{AFMNSR21}]\label{defn: WHD} 
A square matrix is called weakly Hadamard if it satisfies the following two conditions:
\begin{itemize}
    \item The entries of the matrix are from the set $\{-1,0,1\}$.
    \item There is an ordering of the columns of the matrix so that the non-consecutive columns
are orthogonal. 
\end{itemize}
\end{defn}

\begin{defn}[{\cite[Definiton 1.1]{AFMNSR21}}]\label{defn: WHDG}
A graph is weakly Hadamard diagonalizable if its Laplacian matrix can be diagonalized with a weakly Hadamard matrix.
\end{defn}

A large class of block graphs of orthogonal arrays satisfy this definition according to the following theorem.

\begin{thm}[{\cite[Theorem 5.19]{AFMNSR21}}]\label{thm: WHD}
Let $O=OA(m, n)$ be an orthogonal array that can be extended to an orthogonal array with $n+1$ rows. Then its block graph $X_{O}$ is weakly Hadamard diagonalizable.
\end{thm}

\section{Peisert-type graphs as block graphs of orthogonal arrays}\label{sect:peisert-orthogonal-array}

In this section, we first construct an orthogonal array $O_q$ from the affine Galois plane $AG(2,q)$ and then realize every Peisert-type graph as the block graph of some subarray of $O_q$, thereby proving Theorem~\ref{thm: PT=OA}.

For any odd prime power $q$, the field $\mathbb{F}_{q^2}$ can be naturally viewed as the affine Galois plane $AG(2,q) \cong \F_q \times \F_q$. Indeed, once $\alpha\in \F_{q^2}\setminus \F_q$ is picked, we have a bijective map $\F_{q}\oplus \alpha\F_q\to \F_{q^2}$ sending $(x, y)$ to $x+\alpha y$. 

\textbf{Construction.} We construct $O_q \colonequals OA(q+1,q)$ as the \emph{point-line incidence} orthogonal array of the affine Galois plane $AG(2,q)$. The $q+1$ rows of the orthogonal array are indexed by the slopes (directions) $k$ of the lines; note that either $k$ belongs to $\F_q$ or represents the vertical direction. The $q^2$ columns are indexed by the points $(x,y) \in AG(2,q)$. If $k \in \F_q$, we define the entry at the row indexed by the slope $k$ and column indexed by $(x,y)$ as the $y$-intercept of the line with slope $k$ passing through the point $(x,y)$. More precisely, this entry has the value $y-kx \in \F_q$. If $k=\infty$ is the vertical direction, then the entry is simply the $x$-coordinate of the given column $(x, y)$. We claim that this defines an orthogonal array: to see this, pick two rows indexed by $k_1$ and $k_2$, and $c_1,c_2 \in \F_q$. We need to show that there is a unique column indexed by $(x,y)$, such that the corresponding entries in $k_1$-th and $k_2$-th row are $c_1$ and $c_2$, respectively. There are 2 cases to consider:
\begin{itemize}
    \item If $k_1,k_2 \in \F_q$, then there is a unique point $(x,y) \in AG(2,q)$ such that $y=k_1x+c_1$ and $y=k_2x+c_2$.
    \item If $k_1 \in \F_q, k_2=\infty$, then there is a unique point $(x,y) \in AG(2,q)$ such that $y=k_1x+c_1$ and $x=c_2$.
\end{itemize}
The two statements above can be verified algebraically or geometrically using the fact that two non-parallel lines in $AG(2,q)$ intersect at a point. 

A {\em subarray} $O'$ of an orthogonal array $O$ consists of a subset of rows in $O$ with the same columns. By the definition of an orthogonal array, it is clear that every subarray is still an orthogonal array. We realize each Peisert-type graph as the block graph of a subarray of $O_q$ constructed above to prove Theorem~\ref{thm: PT=OA}.

\begin{proof}[Proof of Theorem~\ref{thm: PT=OA}] 
Given a Peisert-type graph $X=\operatorname{Cay}(\F_{q^2}^{+}, S)$ with $S=c_1 \F_q^{\ast} \cup \cdots \cup  c_m \F_{q}^{\ast}$, we first pick $\alpha \in \F_{q^2} \setminus \F_q$ so that $\alpha\notin S$. After identifying $\F_{q^2} = \F_q \oplus \alpha \F_q$ with $AG(2, q)$, we consider the orthogonal array $O_{q}$. After expressing $c_i=u_i + v_i \alpha$ with $u_i, v_i\in \F_q$, and scaling by $u_i^{-1}$, we can assume that $u_i=1$ (note that $u_i\neq 0$ because $\alpha\notin S$.) Pick the rows indexed by $v_i\in \F_q$ to form the subarray $O'$ of $O_q$. 

We claim that $X$ is isomorphic to the block graph $X_{O'}$. Consider the function $f\colon V(X_{O'})\to V(X)$ on the vertex sets defined by sending $(x,y)\in AG(2, q)$ to $x+y\alpha\in \F_{q^2}$. It is clear that $f$ is bijective. Recall that $x_1+y_1\alpha$ and $x_2+y_2\alpha$ are adjacent in $X$ if 
$(y_2-y_1)\alpha+(x_2-x_1) \in c_i\F_q^*$ for some $1 \leq i \leq m$. Since $\alpha\not\in S$, we must have $x_1\neq x_2$, and so we can rewrite the condition as $\frac{y_2-y_1}{x_2-x_1}\alpha+1 \in c_i\F_q^*$. The resulting direction of the line joining $(x_1,y_1)$ and $(x_2,y_2)$ is $\frac{y_2-y_1}{x_2-x_1}=v_i$. This completes the proof of the first assertion.

If $C$ is a canonical clique in $X$, then $C$ is an affine transformation of the subfield $\F_{q}\subset \F_{q^2}$; more precisely, $C=c_i\F_q + b$ for some $1\leq i\leq m$ and $b \in \F_{q^2}$. After writing $b=c+d\alpha$ for $c, d\in\F_q$, we have $C=\{(1+v_i\alpha)t+c+d\alpha: t \in \F_q\}$, which corresponds to the line in $AG(2,q)$ parametrized by $(t+c, v_it+d)$. The equation of the line is $y-d=v_i(x-c)$, or equivalently, $y=v_ix+(d-cv_i)$. While the choice of $b$ is not unique, it always gives rise to the same line. 

On the other hand, a canonical clique in the block graph $X_{O'}$ consists of a set of vertices (columns in $O'$) sharing the same entry along a given row (a slope), that is, a set of vertices lying on a line $y=v_ix+\delta$ for some $1 \leq i \leq m$ and $\delta \in \F_q$. This corresponds to the maximum clique $C=c_i \F_q+\delta \alpha$ in $X$. Thus, there is an explicit equivalence between canonical cliques in $X$ and canonical cliques in the block graph $X_{O'}$ given by sending $c_i \F_q + b$ with $b=c+d\alpha$ to the set of columns containing the entry $d-cv_i$ along the row indexed by the slope $v_i$.
\end{proof}

We now mention a few consequences of Theorem~\ref{thm: PT=OA}. First, the spectrum of a Peisert-type graph can be quickly computed using the known spectrum of the corresponding block graph.

\begin{proof}[Proof of Corollary~\ref{cor: srg}]
This follows from Theorem~\ref{thm: PT=OA} and Theorem~\ref{thm: OAsrg}.
\end{proof}

Next, we show that Peisert-type graphs are weakly Hadamard diagonalizable. 

\begin{proof}[Proof of Theorem~\ref{thm: PTWHD}]
Recall that in the proof of Theorem~\ref{thm: PT=OA}, we showed that each Peisert-type graph is isomorphic to the block graph of a subarray of $O_q$, where $O_q$ is an orthogonal array $OA(q+1,q)$. The conclusion follows from Theorem~\ref{thm: WHD}.
\end{proof}

We can also compute the clique number and the chromatic number of a Peisert-type graph.

\begin{proof}[Proof of Theorem~\ref{thm:peisert-type-chromatic}]
Recall that in the proof of Theorem~\ref{thm: PT=OA}, we showed that $X$ is isomorphic to the block graph of a (proper) subarray $O'$ of $O_q$. In particular, $O'$ is extendible. Thus Theorem~\ref{thm: OAchi} implies the block graph of $O'$ has chromatic number $q$. It follows that $\chi(X)=q$.

We present two different proofs of $\omega(X)=q$. Note that $\F_q$ forms a clique in $X$, and so $\omega(X)\geq q$. It remains to show that $\omega(X)\leq q$. One way is to apply Corollary~\ref{cor: srg} and Lemma~\ref{HoffmanBound}. The second way is to observe that $\omega(X) \leq \chi(X)=q$.
\end{proof}

\begin{rem}
While Peisert-type graphs may not be self-complementary in general, we can use Theorem~\ref{thm:peisert-type-chromatic} to compute the independence number $\alpha(X)$ as follows. If $X=\operatorname{Cay}(\F_{q^2}^+, S)$ is a Peisert-type graph of type $(m, q)$, then $X' = \operatorname{Cay}(\F_{q^2}^+, \F_{q^2}^{\ast}\setminus S)$ is a Peisert-type graph of type $(q+1-m, q)$. Thus, $\alpha(X) = \omega(X') = q$. Since equality $\alpha(X)\omega(X)=q^2=|V(X)|$ holds in the clique-coclique bound and $X$ is vertex-transitive, every maximum coclique meets each maximum clique in exactly one vertex \cite[Theorem 2.1.1]{GM}.
\end{rem}

Another application of Theorem~\ref{thm: PT=OA} is a self-contained proof for the strict-EKR property of certain Peisert-type graphs. 

\begin{proof}[Proof of Corollary~\ref{cor: EKRmsmall}]
This follows from combining Theorem~\ref{thm: PT=OA} and Theorem~\ref{thm: EKROA} on the strict-EKR property of the block graph of the corresponding orthogonal array.
\end{proof}

\section{The EKR-module property of Peisert-type graphs}\label{sect:ekr-module}

Given a graph $\Gamma = (V,E)$, a function $f:V\rightarrow \mathbb{R}$ is called a $\theta$-eigenfunction, if $f \not\equiv 0$ and the equality 
\begin{equation}\label{eigenfunction}
\theta f(\gamma) = \sum\limits_{\delta \in N(\gamma)}f(\delta)   
\end{equation}
holds for every vertex $\gamma \in V$, where $N(\gamma)$ denotes the set of neighbors of $\gamma$.

\subsection{An eigenbasis for the block graph of an orthogonal array} 

We say that a vector in $\mathbb{R}^n$ is \emph{balanced} if it is orthogonal to the all-ones vector \textbf{1}. If $v_A$ is the characteristic vector of a subset $A$ of the set $V$,
then we say that 
$$
v_A - \frac{|A|}{|V|}\textbf{1}
$$
is the \emph{balanced characteristic vector (function)} of $A$. Next, we discuss eigenfunctions of the block graphs of orthogonal arrays and explain how it is spanned by the characteristic vectors of the canonical cliques $S_{r, i}$.

Let $OA(m, n)$ be an orthogonal array with entries from the set $\{1, 2, \ldots, n\}$. The strongly regular graph $X_{OA(m,n)} = (V,E)$ defined as the block graph of $OA(m,n)$ induces $m$ partitions, with each row corresponding to a partition of the vertex set into canonical cliques of size $n$. Denote the partitions by $\Pi_1, \ldots, \Pi_m$. Let $\Pi_r = (S_{r,1}, S_{r,2}, \ldots, S_{r,n})$ be such a partition for some $r \in \{1,\ldots, m\}$. Fix a clique from this partition, say $S_{r,1}$. Given an integer $i \in \{2,\ldots, n\}$, define a function $f_{r,i}:V \rightarrow \mathbb{R}$ as follows.
For a vertex $\gamma \in V$, put
$$
f_{r, i}(\gamma):=
\left\{
  \begin{array}{ll}
     1, & \hbox{if~$\gamma \in S_{r,1}$;} \\
     -1, & \hbox{if~$\gamma \in S_{r,i}$;} \\
     0, & \hbox{otherwise.}
  \end{array}
\right.
$$

\begin{lem}\label{ef}
For each $r \in \{1, \ldots, m\}$, $i \in \{2, \ldots, n\}$, the function $f_{r,i}$ is an eigenfunction of $X_{OA(m,n)}$ corresponding to its largest non-principal eigenvalue $n-m$.
\end{lem}

\begin{proof}
It suffices to check the condition (\ref{eigenfunction}) for the eigenvalue $\theta = n-m$ and the function $f_{r,i}$. 

Let $\gamma$ be a vertex from $S_{r,1}$. The left side of (\ref{eigenfunction}) is equal to $n-m$. On the other hand, $\gamma$ has $n-1$ neighbors in $S_{r,1}$ (with values 1) and $m-1$ neighbors in $S_{r,i}$ (with values $-1$). Thus, the right side of (\ref{eigenfunction}) is equal to $(n-1)-(m-1) = n-m$. In the same manner, the equality (\ref{eigenfunction}) holds for any vertex $\gamma$ from $S_{r,i}$.

Let $\gamma$ be a vertex that is not from $S_{r,1} \cup S_{r,i}$. Then $\gamma$ has $m-1$ neighbors in $S_{r,1}$ (with values $1$) and $m-1$ neighbors in $S_{r,i}$ (with values $-1$). Thus, the equality (\ref{eigenfunction}) holds.
\end{proof}

The following lemma constructs an explicit basis for the $(n-m)$-eigenspace of $X_{OA(m,n)}$.

\begin{lem}\label{MOLSBasis}
The functions $$f_{1, 2}, f_{1, 3}, \ldots, f_{1, n}, f_{2, 2}, f_{2, 3}, \ldots, f_{2, n}, \ldots, f_{m, 2}, f_{m, 3}, \ldots, f_{m, n}$$ form a basis of the eigenspace of $X_{OA(m,n)}$ corresponding to the largest non-principal eigenvalue $n-m$. 
\end{lem}
\begin{proof}
It follows from Theorem~\ref{thm: OAsrg} and two facts. The first fact is that functions $f_{r_1, i_1}$ and $f_{r_2, i_2}$ are orthogonal if $r_1 \ne r_2$. The second fact is that, given $r \in \{1,\ldots, m\}$, the functions $$f_{r,2}, f_{r,3}, \ldots, f_{r,n}$$
are linearly independent.
\end{proof}

The result in the previous lemma also appeared in ~\cite[Theorem 5.5.5]{GM} implicitly; the proof given above is more elementary. In ~\cite[Lemma 15]{GKKSV20}, a similar basis was constructed for the largest non-principal eigenvalue of the Star graphs. We also point out the similarity between the eigenfunctions from Lemma~\ref{ef} and the eigenfunctions of Hamming graphs and Johnson graphs considered in~\cite{V17} and~\cite{VMV18}.

\subsection{Proof of Theorem~\ref{thm: main}} For the rest of the section, let $X$ denote a fixed Peisert-type graph of type $(m, q)$. Recall that in the proof of Theorem~\ref{thm: PT=OA}, we associated to $X$ a subarray $O'$ of the orthogonal array $O_q$. In our construction, both the rows and the entries of $O'$ were indexed by elements of $\F_q$. Since we will use the notation from the previous subsection, we relabel the rows and entries of $O'$ by elements $\{1, \ldots, m\}$ and $\{1, \ldots, q\}$, respectively.

\begin{lem}\label{lem: clique}
Let $C$ be a maximum clique in $X$. Then its balanced characteristic vector $v_C$ lies in the $(q-m)$-eigenspace of $X$.
\end{lem}
\begin{proof}
It suffices to check the condition (\ref{eigenfunction}) for the eigenvalue $\theta = q-m$ and the function $qv_C-\mathbf{1}$.

Take a vertex $\gamma \in C$. The left side of (\ref{eigenfunction}) is equal to $(q-1)(q-m)$.
On the other hand, $\gamma$ has $q-1$ neighbors in $C$  (with values $q-1$) and $m(q-1) - (q-1)=(m-1)(q-1)$ neighbors not in $S$ (with values $-1$). Thus, the right side of (\ref{eigenfunction}) is equal to $(q-1)^2 - (m-1)(q-1)= (q-1)(q-m)$. 

Take a vertex $\gamma \notin C$. The left side of (\ref{eigenfunction}) is equal to $-(q-m)$. Recall that we have shown in Corollary~\ref{cor: srg} that $X$ is a strongly regular graph with parameter $\mu=m(m-1)$. Since $C$ is a maximum clique, Lemma~\ref{HoffmanBound} implies that $C$ is a regular clique. Thus, $\gamma$ has $\mu/m=(m-1)$ neighbors in $C$ (with values $q-1$) and $m(q-1) - (m-1)$ neighbors not in $C$ (with values $-1$). Thus, the right side of (\ref{eigenfunction}) is equal to $(q-1)(m-1) - m(q-1)+ (m-1)= -(q-m)$. 
\end{proof}

Now we are ready to present an explicit basis for the $(q-m)$-eigenspace of a Peisert-type graph $X$ of type $(m, q)$.

\begin{lem}\label{PaleyBasis} 
The $m(q-1)$ functions
$$f_{1,2}, f_{1,3}, \ldots, f_{1,q}, f_{2,2}, f_{2,3}, \ldots, f_{2,q}, \ldots, f_{m,2}, f_{m,3}, \ldots, f_{m,q}$$ form a basis of the eigenspace of $X$ corresponding to the largest non-principal eigenvalue $q-m$. 
\end{lem}
\begin{proof}
This directly follows from Theorem~\ref{thm: PT=OA} and Lemma~\ref{MOLSBasis}.
\end{proof}

Given a partition $\Pi_r = \{S_{r,1}, S_{r,2}, \ldots, S_{r,q}\}$ for some $r \in \{1,\ldots, m\}$ and a positive integer $i \in \{1, \ldots, q\}$, define a function $g_{r,i}: V(X) \rightarrow \mathbb{R}$ by the following rule. For any $\gamma \in V(X)$,
put
$$
g_{r, i}(\gamma):=
\left\{
  \begin{array}{ll}
     q-1, & \hbox{if~$\gamma \in S_{r,i}$;} \\
     -1, & \hbox{otherwise.}
  \end{array}
\right.
$$
Note that the a function $g_{r,i}$ is a $(q-m)$-eigenfunction of $X$, and $g_{r, i}/q$ is equal to the balanced characteristic function of the clique $S_{r,i}$.

\begin{prop}\label{prop:basis-eigenspace}
The $m(q-1)$ functions 
$$
g_{1,2}, g_{1,3}, \ldots, g_{1,q}, g_{2,2}, g_{2,3}, \ldots, g_{2,q}, \ldots, g_{m,2}, g_{m,3}, \ldots, g_{m,q}
$$
form a basis of the the eigenspace of $X$ corresponding to the eigenvalue $q-m$.
\end{prop}
\begin{proof}
We first show that the functions $g_{1,1}, g_{1,2}, \ldots, g_{1,q}, g_{2,1}, g_{2,2}, \ldots, g_{2,q}, \ldots, g_{m,1}, g_{m,2}, \ldots, g_{m,q}$ span the eigenspace of $X$ corresponding to the eigenvalue $q-m$. This follows from Lemma~\ref{PaleyBasis} and the fact that, for any $r \in \{1, \ldots, m\}$ and $i \in \{2, \ldots, q\}$, the equality
$$
f_{r,i} = \frac{1}{q} (g_{r,1} - g_{r,i})
$$
holds. Next, we find a linearly independent subset which still spans the eigenspace. Note that, for any $r \in \{1, \ldots, m\}$, the equality 
$$
g_{r,1} + g_{r,2} + \ldots + g_{r,q} = 0
$$
holds. It means that the $m(q-1)$ functions 
$$
g_{1,2}, g_{1,3}, \ldots, g_{1,q}, g_{2,2}, g_{2,3}, \ldots, g_{2,q}, \ldots, g_{m,2}, g_{m,3}, \ldots, g_{m,q}
$$
still span the $(q-m)$-eigenspace of $X$ and thus form a basis since the $(q-m)$-eigenspace has dimension $m(q-1)$ by Corollary~\ref{cor: srg}.
\end{proof}

\begin{proof}[Proof of Theorem~\ref{thm: main}] 
Let $C$ be a maximum clique in $X$. By Lemma~\ref{lem: clique} and Proposition~\ref{prop:basis-eigenspace}, the balanced characteristic vector of $C$ is a linear combination of the balanced characteristic vectors of the canonical cliques in $X$. However, note that the sum of all characteristic vectors of the canonical cliques is a constant multiple of the all-ones vector due to symmetry. It follows that the characteristic vector of $C$ is a linear combination of the characteristic vectors of the canonical cliques.
\end{proof}

\section{Counterexamples for the strict-EKR property}\label{sect:counterexamples}

In this section, we focus on Peisert-type graphs which fail to satisfy the strict-EKR property. After constructing an infinite family of such graphs in Section~\ref{subsect:counterexamples-infinite}, we give details of a concrete counterexample involving a generalized Peisert graph.

\subsection{Proof of Theorem~\ref{thm: counterexample}}\label{subsect:counterexamples-infinite}

The construction of the following counterexamples is inspired by \cite[Section 5.3]{AY2}, \cite[Theorem 1.2]{AY}, and \cite[Example 2.18]{AY}.

\begin{proof}[Proof of Theorem~\ref{thm: counterexample}]
Let $q$ be an odd prime power which is not a prime. Let $g$ be a primitive root of $\F_{q^2}$. We pick a proper subfield $K$ of $\F_q$. Let $t$ be the dimension of $\F_q$ as a $K$-vector space. Consider the following $K$-subspace of $\F_{q^2}$:
$$C=\bigoplus_{j=0}^{t-1} g^j K.$$
Then $|C|=q$ and $C\setminus \{0\}$ is the union of $\frac{q-1}{|K|-1}$ $K^*$-cosets; in particular, $C\setminus \{0\}$ is contained in the union $S$ of \emph{exactly} $\frac{q-1}{|K|-1} \leq \frac{q-1}{2}$ many distinct $\F_q^*$-cosets because $g^i\F_q^* \neq g^j \F_q^*$ for $0 \leq i<j \leq t-1$ and $t\leq \log_2(q)$. 

The Peisert-type graph $X=\operatorname{Cay}(\F_{q^2}^+, S)$ contains $C$ as a maximum clique; note that $C$ is indeed closed under subtraction, that is, $C-C=C\subset S\cup \{0\}$ because $C$ is a subspace. We claim that $C$ is a non-canonical clique in $X$. Assuming that $C$ is canonical, it must be the image of the subfield $\F_q$ under a linear transformation; since $g\in C$ and $|C|=q$, this would imply that $C=g\F_q$ which is a contradiction as $1\in C$ but $g^{-1}\notin \F_q$. Consequently, $X$ does not have the strict-EKR property.

In the case when $q$ is a square, the above construction could be simplified by taking $K$ to be $\F_{\sqrt{q}}$. The resulting clique is $C=\F_{\sqrt{q}} \oplus g \F_{\sqrt{q}}$ inside a Peisert-type graph $X=\operatorname{Cay}(\F_{q^2}^+,S)$, where the connection set $S$ is the union of exactly $\frac{q-1}{|K|-1}=\sqrt{q}+1$ many $\F_q^{\ast}$-cosets. Thus, the corresponding Peisert-type graph $X$ is of type $(\sqrt{q}+1, q)$ and contains $C$ as a non-canonical clique. This shows that the ``$q>(m-1)^2$'' condition in Corollary~\ref{cor: EKRmsmall} is sharp when $q$ is a square. 
\end{proof}

\subsection{Detailed analysis of a counterexample of small size}\label{subsect:counterexamples-detailed-analysis}
We proceed to give a detailed analysis of a Peisert-type graph with order $81$ which fails to satisfy the strict-EKR property. 

We consider the generalized Peisert graphs $GP^{\ast}(3^4,3^2+1)=GP^{\ast}(81,10)$. While the connection set of the graph depends on the choice of the primitive root, the isomorphism class is unique.
Let $a$ be a primitive root of $\mathbb{F}_{81}$ satisfying $a^4-a^3-1=0$. Recall that 
$$
GP^*(81,10)=\operatorname{Cay}\big(\F_{81}^+, \F_9^* \cup a\F_9^* \cup a^2 \F_9^* \cup a^3 \F_9^* \cup a^4 \F_9^*\big).
$$
Using SageMath, we see that the graph $GP^{\ast}(81,10)$ has clique number $9$, and has $9$ maximum cliques containing $0$. The $5$ canonical cliques containing $0$ can be read from the connection set:
\begin{align*}
    \mathbb{F}_9,  a\mathbb{F}_9, a^2\mathbb{F}_9, a^3\mathbb{F}_9, a^4\mathbb{F}_9.
\end{align*}
There are exactly $4$ non-canonical cliques containing $0$. These are:
\begin{align*}
    C_1 = \langle 1, a^3\rangle, \ \ \ \ \ & C_2 =\langle a, a^{10}\rangle,\\ 
    C_3 = \langle a^{11}, a^{20}\rangle, \ \ \ \ \ & 
    C_4 =\langle a^{30}, a^{33} \rangle,
\end{align*}
where $\langle a^i, a^j\rangle$ stands for the $\F_3$-subspace generated by $a^i$ and $a^j$.

Let $v_1, v_2, ..., v_{40}$ be the balanced characteristic of the canonical cliques \emph{not} containing $0$. By Proposition~\ref{prop:basis-eigenspace}, these $40=\frac{81-1}{2}$ vectors form a basis for the eigenspace $W$ of $GP^*(81,10)$ corresponding to the eigenvalue $q-m=9-5=4$. The following table lists these $40$ canonical cliques of $GP^*(81,10)$. The entry $(s, t)$ is a shorthand for the canonical clique $s\F_{9}+t$.  

\begin{table}[ht!]
\begin{tabular}{ |c|c|c|c|c| } 
 \hline 
$(1, a)$ & $(a, a^2)$  & $(a^2, 1)$  &  $(a^3, a)$ &  $(a^4, a)$ \\
 \hline 
$(1, 2a)$ &  $(a, 2a^2)$ & $(a^2, 2)$  &  $(a^3, 2a)$  &  $(a^4, 2a)$ \\
 \hline 
\cellcolor[HTML]{C0C0C0}\ $(1, a^2)$ & \cellcolor[HTML]{C0C0C0}\ $(a, 1)$  &  $(a^2, a)$ & \cellcolor[HTML]{C0C0C0}\ $(a^3,1)$ & \cellcolor[HTML]{C0C0C0}\ $(a^4,1)$
 \\
  \hline 
\cellcolor[HTML]{C0C0C0}\ $(1, 2a^2)$ & \cellcolor[HTML]{C0C0C0}\ $(a, 2)$ & $(a^2, 2a)$ & \cellcolor[HTML]{C0C0C0}\ $(a^3,2)$ &  \cellcolor[HTML]{C0C0C0}\ $(a^4,2)$ \\
 \hline 
\cellcolor[HTML]{C0C0C0}\ $(1, a^2+a)$ & \cellcolor[HTML]{C0C0C0}\ $(a, a^2+1)$ & $(a^2, a+1)$ &  \cellcolor[HTML]{C0C0C0}\ $(a^3,a+1)$ &  \cellcolor[HTML]{C0C0C0}\ $(a^4,a+1)$\\
 \hline 
\cellcolor[HTML]{C0C0C0} $(1, a^2+2a)$ & \cellcolor[HTML]{C0C0C0}\  $(a, a^2+2)$  &  $(a^2, a+2)$ &  \cellcolor[HTML]{C0C0C0}\ $(a^3,a+2)$  & \cellcolor[HTML]{C0C0C0}\ $(a^4,a+2)$ \\
 \hline 
\cellcolor[HTML]{C0C0C0}\ $(1, 2a^2+a)$ & \cellcolor[HTML]{C0C0C0}\ $(a, 2a^2+1)$ & $(a^2, 2a+1)$ & \cellcolor[HTML]{C0C0C0}\ $(a^3,2a+1)$ & \cellcolor[HTML]{C0C0C0}\ $(a^4,2a+1)$ \\
 \hline 
\cellcolor[HTML]{C0C0C0}\ $(1, 2a^2+2a)$ & \cellcolor[HTML]{C0C0C0}\ $(a, 2a^2+2)$  &  $(a^2, 2a+2)$ & \cellcolor[HTML]{C0C0C0}\ $(a^3, 2a+2)$ & \cellcolor[HTML]{C0C0C0}\ $(a^4,2a+2)$ \\
 \hline 
\end{tabular}
\end{table}

Let $v$ be the balanced characteristic function of the non-canonical clique $C_2 = \langle a, a^{10} \rangle$. We can express $v = b_1 v_1 + b_2 v_2 + ... + b_{40} v_{40}$ for some $b_i\in\mathbb{Q}$. Using SageMath, we found that among these 40 coefficients, $16$ of them are $0$ and $24$ of them are $-1/3$. The value of $b_i$ is $-1/3$ if the corresponding cell is shaded grey, and it is $0$ if the cell is unshaded. In other words, $v$ has the ``distribution" given by $(0^{16}; (-1/3)^{24})$.

Note that there are $36$ non-canonical cliques in $GP^{\ast}(81,10)$. Using SageMath, we found that exactly $4$ non-canonical cliques have the distribution $(0^{16}; (-1/3)^{24})$, and these are precisely the $4$ non-canonical cliques containing $0$. The remaining $32$ non-canonical cliques have the distribution equal to $(0^{25}; (-1/3)^{6}; (1/3)^{9})$.

\section{Open problems}\label{sect:open-problems}
We end the paper by presenting two open problems. 

Inspired by our example in Section~\ref{subsect:counterexamples-detailed-analysis}, it is natural to investigate how the coefficients for the balanced characteristic function of non-canonical cliques vary in a given Peisert-type graph. 

\begin{question}
Let $X$ be a Peisert-type graph of type $(m, q)$. Fix a basis of the $(q-m)$-eigenspace of $X$ according to Proposition~\ref{prop:basis-eigenspace}. Then do any two non-canonical cliques containing $0$ share the same distribution of coefficients in the given basis? Moreover, do any two non-canonical cliques not containing $0$ share the same distribution of coefficients in the given basis?
\end{question}

Another problem, motivated by the counterexamples found in Theorem~\ref{thm: counterexample}, is the following.

\begin{prob}
Characterize Peisert-type graphs with the strict-EKR property.
\end{prob}

According to Theorem~\ref{thm: PT=OA}, this is a special case of another open problem: characterize orthogonal arrays whose block graphs satisfy the strict-EKR property \cite[Problem 16.4.1]{GM}. We believe that this is an interesting sub-problem to explore.

\section*{Acknowledgments}
During the preparation of this work, S. Asgarli was supported by a postdoctoral research fellowship from the University of British Columbia and NSERC PDF grant. S. Asgarli thanks Zinovy Reichstein for his comments on the manuscript. S. Goryainov is grateful to Alexander Gavrilyuk and Alexandr Valyuzhenich for valuable discussions concerning the paper. H. Lin is supported by National Natural Science Foundation of China (Nos. 11771141 and 12011530064). C. H. Yip is supported by a Four Year Doctoral Fellowship from the University of British Columbia. C. H. Yip thanks Shaun Fallat for the clarification of a theorem in \cite{AFMNSR21}. The authors are also grateful to the referees for valuable comments, corrections, and suggestions.

\bibliographystyle{abbrv}
\bibliography{biblio}

\end{document}